\documentclass[11pt]{article}
\usepackage{mathrsfs}
\usepackage{esdiff}
\usepackage{amsthm,amsfonts}
\usepackage{amsmath}
\usepackage{amssymb}
\usepackage{graphics,epsfig}
\usepackage{stackrel}

\newtheorem{lemma}{Lemma}[section]

\newenvironment{remark}{\addtocounter{theorem}{1}\vskip 0.2cm{\sc Remark
\thetheorem.}}{\hfill\vskip 0.2cm}

\newcommand{\beq}{\begin{equation}}
\newcommand{\enq}{\end{equation}}

\newcommand{\beqa}{\begin{eqnarray}}
\newcommand{\enqa}{\end{eqnarray}}
\newcommand{\beqas}{\begin{eqnarray*}}
\newcommand{\enqas}{\end{eqnarray*}}

\newtheorem{thm}{Theorem}[section]
\newtheorem{lem}{Lemma}[section]

\newtheorem{cor}{Corollary}[section]

\numberwithin{equation}{section}
\def\lbr{\left}
\def\rbr{\right}

\begin{document}

\title{The Boundary Value Problem for a Static 2D Klein--Gordon Equation in the Infinite Strip and in the Half-Plane\thanks{I am grateful to Prof. Yu.M. Kabanov and to Dr. M. Zhitlukhin for important discussions and help}}
\author{Dmitry Muravey\thanks{dmuravey@hse.ru} \\
International Laboratory of Quantitative Finance, \\
National Research University  Higher School of Economics,\\   Moscow, Russia}
\maketitle

\abstract{
We provide explicit formulas for the Green function of an elliptic PDE in the infinite strip and the  half-plane. They are expressed in  elementary and special functions. Proofs of uniqueness and existence are also given.
}

\section{Motivation}
\par
This research is motivated by the multi-asset American option pricing problem(see e.g. \cite{Broadie}). Using the standard arbitrage theory framework it can be shown that the option prices are bounded by solutions of elliptic PDE problems. These problems have specific features:  the domains are unbounded and the boundary functions are not smooth and do not vanish at infinity. We  reduce the original PDE determining the price to a static Klein--Gordon equation (SKGE) in an  infinite strip (for the alternative dual options) and in the half plane (the exchange and basket options). The boundary conditions  of the corresponding  problems are bounded  H\"older functions. This leads to the following questions:
\begin{enumerate}
\item  Can we explicitly solve the boundary problems for the SKGE?

\item If this is the case, can we construct computer-
friendly representations?

\item  Is the obtained solution classical?

\item Is it unique?
\end{enumerate}
To our knowledge,  these aspects of the mentioned problems have not been studied in the literature. Even the answers to the two last questions seem not be  obvious.  The majority  of the  known results deal with bounded domains and one can not directly apply available theorems to the equations in the infinite strip and in the half-plane. Also, if the boundary conditions are not  twice differentiable, then  the smoothness of  the solution is not clear.
In this paper we give answers to  all four questions.  Our results are as follows:

\begin{enumerate}
\item  The boundary value problem in the infinite strip and in the half-plane allows a closed form solution.
\item   Solutions can be represented in terms of elementary functions or in terms of special functions.
\item   For all bounded boundary functions with H\"older property the solution is classical.
\item   The solution is unique.
\end{enumerate}

We could find in the literature only a few related works. The paper \cite{Danilov} contains  closed form formulas of boundary problem  for the Laplace equation in the infinite strip. In the recent paper \cite{Melnikov} there is a  variety of Green functions for the homogeneous Poisson problems for SKGE in some unbounded domains, including the strip and the half-plane.
\par
The paper is organized as follows. In Section \ref{Sec: Strip} we obtain explicit formulas for the Green function in the case of infinite strip  and formulate the main theorem on existence, uniqueness and closed form of the solution. In Section \ref{Sec: Plane} we address the case of half-plane domain. All proofs are given in Section \ref{Sec: proofs}.

\section{Problem in the infinite strip}
\label{Sec: Strip}
\subsection{Existence and uniqueness theorem}
Let $\Pi^\pi=\mathbb{R} \times [0,\pi]=\left\{ (x,y) \in \mathbb{R}^2,\; y \in [0,\pi] \right\}$.  Let $LV=\Delta V - r^2 V$ be an operator acting on the twice differentiable functions $V=V(x,y)$  defined in the interior of $\Pi^\pi$.  Here $\Delta$ is the Laplacian and $r \in \mathbb{R}$.
We consider the boundary value problem
\begin{equation}
\label{Eq: Main Strip}
    \left\{ {\begin{array}{l}
        LV(x,y) = 0, \quad (x,y) \in {\rm int\,} \Pi^\pi, \\
        V(x,0) = \varphi(x), \quad x \in \mathbb{R},  \\
        V(x,\pi) = 0, \quad x \in \mathbb{R}.  \\
    \end{array}} \right.
\end{equation}
The following theorem claims that under certain assumptions the problem (\ref{Eq: Main Strip}) admits  a unique classical solution and provides an explicit form for it.

\begin{thm}
\label{THM: Strip}
Let $H^\lambda$ be the space of H\"older  functions of order $\lambda >0$ and let $\varphi$  be a bounded function from $H^\lambda$.Then the solution of the problem (\ref{Eq: Main Strip}) exists in the classical sense,  is unique, and allows the  representation
\beq
\label{Eq: Convolution for strip}
V(x,y) = \int_{\mathbb{R}} \varphi(u) G^\pi(x-u,y)du,
\enq
with the Green function
\beq
\begin{split}
\label{Eq: Green function Strip}
G^{\pi}(x,y) = \delta(x)\Theta(-y)+\frac{1}{\pi} \sum_{k=1}^{\infty}\frac{k \sin ky }{\sqrt{k^2+r^2}} e^{-|x|\sqrt{k^2+r^2}}.
\end{split}
\enq
where $\delta(x)$ is the Dirac delta-function and $\Theta(x)=I_{(0,\infty)}(x)$.

\end{thm}

\begin{cor}
The Green function $G^{\pi}_{\Delta}$ for the Laplace equation $\Delta V = 0$ (i.e. for $r=0$) has the  representation (see \cite{Danilov}):
\beq
G^{\pi}_{\Delta}(x,y) = \delta(x)\Theta(-y)+\frac{1}{2\pi} \frac{\sin y}{ \cosh x -\cos y}.
\enq
\end{cor}

\begin{thm}
\label{THM: Strip 2}
The Green function $G^{\pi}$ given by (\ref{Eq: Green function Strip}) has the integral representation
\beq
\label{Eq: Green function Strip Alterantive}
G^{\pi}(x,y) = \delta(x)\Theta(-y)+  \frac{ |x|}{2\pi} \sin y \int_{1}^{\infty}
\frac{J_0 (|x| r \sqrt{t^2 -1 }) \sinh |x|t} {\lbr( \cosh xt - \cos{y}\rbr)^2} dt
\enq
where $J_0(z)$ is the Bessel function of zero order.
\end{thm}

\begin{cor}
The Green function (\ref{Eq: Green function Strip})  can be represented in terms of the Green function $G^{\pi}_{\Delta}$ as follows:
\beq
G^{\pi} (x,y) =
G^{\pi}_{\Delta}(x,y) - r \int_{|x|}^{\infty} G^{\pi}_{\Delta} (t,y) \frac{J_{1} (r\sqrt{t^2 - x^2})t}{\sqrt{t^2 - x^2}}dt, \enq
where $J_1(z)$ is the Bessel function of first order.
\end{cor}

\begin{remark}
The problem $LV=0$ with boundary conditions $V(x,0) = 0$, $V(x,\pi) = \tilde{\varphi}(x)$ can be reduced to the  problem (\ref{Eq: Main Strip} by the substitution $\tilde{y} = \pi -y$. Using the linearity of $L$ we obtain the solution of the problem $LV=0$ with the boundary conditions $V(x,0) = \varphi(x)$ and $V(x,\pi) = \tilde\varphi(x)$ as the sum of solutions of the problems in which one of the boundary conditions is a function equal to zero.
\end{remark}

\subsection{Green function: construction}
In this subsection we derive the formulas (\ref{Eq: Green function Strip}) for the Green function. To this aim we introduce the interaction potential $P^{\pi} (x,y;u)$ defined as the limit
    \[
    P^{\pi} (x,y;u) = \lim_{ \varepsilon \rightarrow 0} P_{\varepsilon}^{\pi} (x,y;u), \quad
    \]
The function $P_{\varepsilon}^{\pi} (x,y;u)$ is the solution in the distribution sense (see, e.g. \cite{Vladimirov}) of the  boundary value  problem
\begin{equation}
\label{Eq: Potential P_0_eps and P_1_eps}
    \left\{ {\begin{array}{l}
        LP_{\varepsilon}^{\pi} = 0,\\
        P_{\varepsilon}^{\pi}|_{y=0} = e^{-\varepsilon (x-u)} \Theta(x-u), \\
        P_{\varepsilon}^{\pi}|_{y=\pi} = 0.  \\
    \end{array}} \right.
\end{equation}
We define the Green function as the partial derivative in $x$ of  $P^{\pi}$:
\beq
G(x-u,y) = \diffp{P^{\pi}(x,y;u)}{x}.
\enq
Let us consider
 the Fourier transform in  $x$ of the potential $P_{\varepsilon}^{\pi}$:
\[
v_{\varepsilon}(y;\xi,u) = \int_{\mathbb{R}} e^{i\xi x} P_{\varepsilon}^{\pi}(x,y;u) dx,
\]
It  solves, as a function of $y$, the two-point   problem
\begin{equation}
\label{Eq: v_0}
    \left\{ {\begin{array}{l}
        \diff[2]{v_{\varepsilon}}{y} -(\xi^2+r^2)v_{\varepsilon} =0,\\
        v_{\varepsilon}(0;\xi,u) = \frac{e^{i\xi u}}{\varepsilon - i\xi},  \\
        v_{\varepsilon}(\pi;\xi,u) =0.  \\
    \end{array}} \right.
\end{equation}
The solution has the form
\[
v_{\varepsilon} (y;\xi,u) = \frac{e^{i\xi u}
\sinh\lbr( (\pi - y) \sqrt{r^2+\xi^2}\rbr)}{(-i\xi+\varepsilon)
\sinh\lbr( \pi \sqrt{r^2+\xi^2}\rbr)}.
\]
Making the inverse transform, we obtain the explicit formula for the potential $P^{\pi}_{\varepsilon}$:
\beq
\label{Eq: inv FFT P0}
P_{\varepsilon}^{\pi} (x,y;u) =\frac{1}{2\pi}\int_{\mathbb{R}} \frac{e^{i\xi (u - x)}
\sinh\lbr( (\pi - y) \sqrt{r^2+\xi^2}\rbr)}{(-i\xi+\varepsilon) \sinh\lbr( \pi \sqrt{r^2+\xi^2}\rbr)} d\xi.
\enq
The integrand  is an analytical function in the whole complex plane except zeros of functions
\[
\xi+i\varepsilon=0, \quad\quad  \sinh\lbr(\pi\sqrt{\xi^2+r^2}\rbr)=0,
\]
that is except the points
\[
\xi = -i\varepsilon, \quad \xi^{\pm}_k =\pm\sqrt{k^2+r^2}.
\]

\begin{figure}
\begin{center}
    \resizebox*{12cm}{!}{\includegraphics{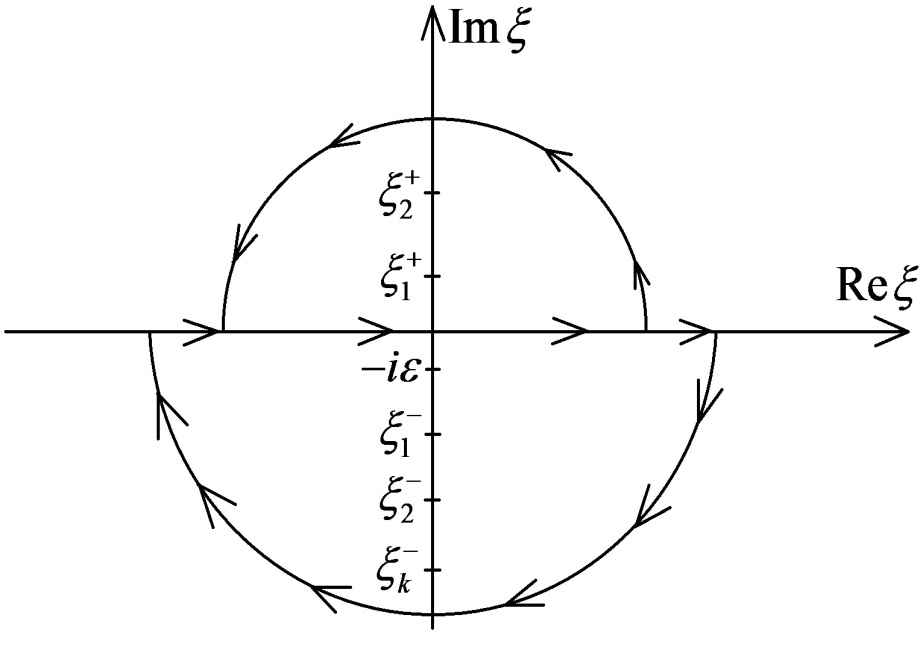}}
    \caption{\label{fig1} Integration loop for the strip case
    \label{fig: strip}}
\end{center}
\end{figure}
Integrating along the loop depicted in Figure \ref{fig: strip} and applying the Jordan lemma we obtain the infinite sum representation of (\ref{Eq: inv FFT P0})
\begin{eqnarray*}
P_{\varepsilon}^{\pi} (x,y;u) &=& i \Theta(x-u) \sum_{k=1}^{\infty} {\rm Res}\, \Phi_{\varepsilon}(\xi_k^+; x,y)\\
&&+i\Theta(u - x)
\lbr(-{\rm Res}\, \Phi_{\varepsilon}(-i\varepsilon;x,y) - \sum_{k=1}^{\infty} {\rm Res}\, \Phi_{\varepsilon}(\xi_k^-;x,y)\rbr).
\end{eqnarray*}
Here $\Phi_{\varepsilon}(\xi;x,y)$ is integrand function in (\ref{Eq: inv FFT P0}). Computing the residuals, we get
\[
{\rm Res}\, \Phi_{\varepsilon}(\xi_k^{\pm}; x,y) =\mp\frac{ i k \sin(ky) e^{\mp(x-u)\sqrt{r^2+k^2}}}{\pi (r^2 + k^2)}, \qquad
{\rm Res}\, \Phi_{\varepsilon}(-i\varepsilon; x,y) = \frac{\sinh\lbr((\pi-y) r \rbr)}{\sinh(\pi r)}.
\]
Letting $\varepsilon \rightarrow 0$ and calculating the partial derivative $\partial P^{\pi} / \partial x$, we can obtain the result.

\subsection{The general elliptic operator}
Now we extend our formula for the strip $\Pi^l$ of width $l$ and more general operator
\[
\label{Eq: Main Strip General}
    \left\{ {\begin{array}{l}
        \sigma_1^2 \diffp[2]{V}{x} + 2\rho \sigma_1 \sigma_2 \frac{\partial^2 V}{\partial x \partial y} +\sigma_2^2 \diffp[2]{V}{y}
        +\alpha_1 \diffp{V}{x} + \alpha_2 \diffp{V}{y}- r^2 V = 0, \quad (x,y) \in {\rm int\,} \Pi^l, \\
        V|_{y=0} = \varphi(x), \quad x \in \mathbb{R},  \\
        V|_{y=\pi} = 0, \quad x \in \mathbb{R},  \\
    \end{array}} \right.
\]
where the coefficient $\rho$ satisfies the condition $|\rho| < 1$. The solution of the problem (\ref{Eq: Main Strip General}) also allows the representation $V= \varphi * \hat{G^l}$ where the Green function $\hat{G^l}$ is given as follows:
\[
\hat{G^{l}}(x,y) = \delta(s(x
,y))\Theta(-y) + \frac{\pi \sigma_2^2}{(1-\rho^2) \sigma_1^2 l^2}
e^{-\frac{\alpha_2 y}{2\sigma^2_2} +\beta s(x,y) } \sum_{k=0}^{\infty} k \sin \lbr(\frac{ \pi y}{l}\rbr)\frac{e^{-R(k)|s(x,y)|} }{R(k)},
\]
where
\[
s(x,y) = \frac{\rho \sigma_1}{\sigma_2}y -x, \qquad
\beta = \frac{1}{2(1-\rho^2)\sigma_1} \lbr( -\frac{\rho \alpha_2}{\sigma_2}+\frac{\alpha_1}{\sigma_1}\rbr),
\]
\[
R(k) = \frac{1}{2(1-\rho^2)\sigma_1}\sqrt{\frac{\alpha_1^2}{\sigma_1^2} -2 \frac{\rho \alpha_1 \alpha_2}{\sigma_1\sigma_2} +\frac{\alpha_2^2}{\sigma_2^2} + 4(1-\rho^2)r^2 + \frac{4(1-\rho^2)\sigma_2^2\pi^2 k^2}{l^2}}.
\]
The proof of uniqueness and existence of the solution and the construction of the Green function is a straightforward extension of arguments  given in  Section \ref{Sec: proofs} for the Laplace operator.

\section{Problem on the half-plane}
\label{Sec: Plane}
\subsection{Existence and uniqueness theorem}
In this section we study  the case of the half-plane $\Pi^\infty = {\mathbb{R} \times \mathbb{R}_+}$. Let us consider the elliptic  boundary value problem for the operator $LV = \Delta V - r^2 V$:
\begin{equation}
\label{Eq: Main HalfPlane}
    \left\{ {\begin{array}{l}
        LV(x,y) = 0, \quad (x,y) \in {\rm int\,} \Pi^\infty, \\
        V|_{y=0} = \varphi(x), \quad x \in \mathbb{R},  \\
        V|_{y=+\infty} = 0, \quad x \in \mathbb{R}.  \\
    \end{array}} \right.
\end{equation}

\begin{thm}
\label{THM: Plane}
Let $H^\lambda$ be the space of H\"older  functions of order $\lambda >0$ and let $\varphi$  be a bounded function from $H^\lambda$.Then the solution of the problem (\ref{Eq: Main HalfPlane}) exists in the classical sense,  is unique, and allows the  representation
\beq
\label{Eq: Convolution for HalfPlane}
V(x,y) = \int_{\mathbb{R}} \varphi(x) G^{\infty}(x-u,y)du.
\enq
where the Green function $G^{\infty}$ has the  form
\beq
\label{Eq: Green function HalfPlane}
G^{\infty}(x,y) = \delta(x)\Theta(-y)+\frac{1}{2\pi} \int_{0}^{\infty}\frac{\xi \sin(\xi y) e^{-|x|\sqrt{\xi^2+r^2}}}{\sqrt{\xi^2+r^2}}d\xi.
\enq
\end{thm}

\begin{cor}
The Green function $G^{\infty}$ can be represented in terms of the modified Bessel function $K_1(z)$ of the first order
\beq
\label{Eq: Green function HalfPalne Alterantive}
G^{\infty} (x,y) = \delta(x)\Theta(-y)+ \frac{r y}{\pi \sqrt{x^2 + y^2}} K_1 \lbr(r \sqrt{x^2 +y^2} \rbr).
\enq
\end{cor}
\begin{proof}
Use the following identity for the Bessel functions (see \cite{Gradshteyn}, 3.914)
\beq
\int_{0}^{\infty} \frac{x \sin(ax)e^{-\beta\sqrt{\gamma^2+x^2}} }{\sqrt{\gamma^2+x^2}} dx = \frac{a\gamma}{\sqrt{a^2+\beta^2}} K_1\lbr( \gamma \sqrt{a^2+\beta^2}\rbr), \qquad {\rm Re}\,\beta, \, {\rm Re}\,\gamma,\,a>0.
\enq
\end{proof}

\subsection{Construction of the Green function}
For this domain we use the similar construction of the interaction potential $P^{\infty}(x,y;u)$ as the limit of  solutions $P_{\varepsilon}^{\infty}(x,y;u)$ to the boundary value problems
\begin{equation}
\label{Eq: Potential Plane}
    \left\{ {\begin{array}{l}
        \Delta P_{\varepsilon}^{\infty} - r^2 P_{\varepsilon}^{\infty} =0,\\
        P_{\varepsilon}^{\infty}|_{y=0} = e^{-\varepsilon(x-u)} \Theta(x-u),\\
        P_{\varepsilon}^{\infty}|_{y=\infty} = 0.
    \end{array}} \right.
\end{equation}
Let us consider the Fourier transform
\[
v_{\varepsilon}(y;\xi,u) = \int_{\mathbb{R} }e^{-i\xi x}P_{\varepsilon}^{\infty}(x,y;u)dx.
\]
Then $v_\varepsilon$ as a function of $y$ solves the boundary value problem for the ODE
\begin{equation}
\label{Eq: v_eps}
    \left\{ {\begin{array}{l}
        \diff[2]{v_{\varepsilon}}{y} -(\xi^2+r^2)v_{\varepsilon} =0,\\
        v_{\varepsilon}(0;\xi,u) =\frac{e^{i\xi u}}{\varepsilon - i\xi},  \\
        v_{\varepsilon}(\infty;\xi,u) =0.  \\
    \end{array}} \right.
\end{equation}
It can be expressed explicitly:
\[
v_\varepsilon(y;\xi,u) = \frac{e^{i\xi u - y\sqrt{\xi^2+r^2}}}{\varepsilon -i\xi}.
\]
As in the previous case we calculate the inverse Fourier transform and obtain the formula for the potential $P_{\varepsilon}^{\infty}$:
\[
P_{\varepsilon}^{\infty} (x,y;u) = \frac{1}{\pi} \int_{\mathbb{R}} \frac{e^{i\xi u - y\sqrt{\xi^2+r^2}}}{\varepsilon -i\xi} d\xi
\]
Integrating along the loop depicted in Figure \ref{fig: plane} and applying the Jordan lemma we get  the following representation for $P_{\varepsilon}^{\infty}$:
\begin{figure}
\begin{center}
    \resizebox*{12cm}{!}{\includegraphics{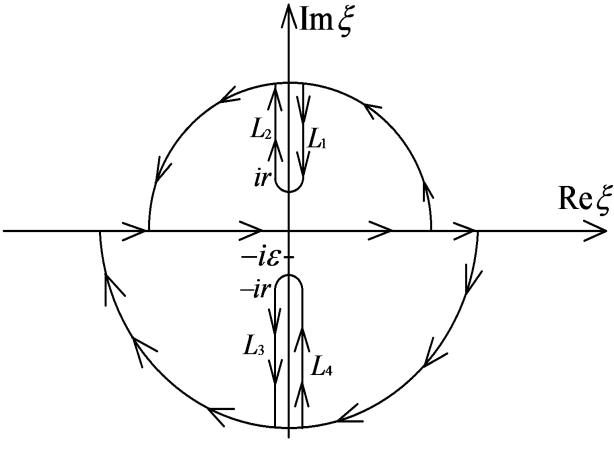}}
    \caption{\label{fig1} The integration loop for the half-plane case
    \label{fig: plane}}
\end{center}
\end{figure}

\begin{eqnarray*}
P_{\varepsilon}^{\infty} (x-u,y) &=& -\frac{\Theta(u-x)}{2 \pi} \int_{L_1+L_2} \Phi(\xi; x,y) d\xi - \frac{\Theta(x-u)}{2 \pi} \int_{L_3+L_4} \Phi(\xi; x,y) d\xi\\
&&
- i\Theta(x-u){\rm Res}\,(\Phi(-i\varepsilon)).
\end{eqnarray*}
In contrast to the strip case we have only one residue $\xi = -i\varepsilon$ and two branch points $\xi =\pm ir$. Letting $\varepsilon\to \infty$ and computing the partial derivative with respect to $x$,  we obtain the needed formula for the Green function.

\subsection{The general elliptic operator}
As in the previous section we generalize the results for a more general operator
\begin{equation}
\label{Eq: Half plane general}
    \left\{ {\begin{array}{l}
        \sigma_1^2 \diffp[2]{V}{x} + 2\rho \sigma_1 \sigma_2 \frac{\partial^2 V}{\partial x \partial y} +\sigma_2^2 \diffp[2]{V}{y}
        \alpha_1 \diffp{V}{x} + \alpha_2 \diffp{V}{y}- r^2 V = 0, \quad (x,y) \in {\rm int\,}\Pi^\infty, \\
        V|_{y=0} = \varphi(x), \quad x \in \mathbb{R},  \\
        V|_{y=\pi} = 0, \quad x \in \mathbb{R},  \\
    \end{array}} \right.
\end{equation}
where $|\rho| < 1$. The solution has the representation $V= \varphi * \hat{G}^{\infty}$ where $\hat{G}^{\infty}$ is the Green function
\[
\hat{G}^{\infty}(x-u,y) = \delta(s(x-u,y))\Theta(-y) +\frac{\sigma_2^2 e^{-\frac{\alpha_2 y}{2\sigma^2_2} +\beta s(x-u,y) } }{\pi (1-\rho^2) \sigma_1^2 }  \int_{0}^{\infty} \xi \sin \lbr(\xi y\rbr)\frac{e^{-R_{\infty}(\xi)|s_0(x-u,y)|} }{R_{\infty}(\xi)}d\xi,
\]
where
\[
R_{\infty}(\xi) = \frac{1}{2(1-\rho^2)\sigma_1}\sqrt{\frac{\alpha_1^2}{\sigma_1^2} -2 \frac{\rho \alpha_1 \alpha_2}{\sigma_1\sigma_2} +\frac{\alpha_2^2}{\sigma_2^2} + 4(1-\rho^2)r^2 + 4(1-\rho^2)\sigma_2^2  \xi^2}.
\]

\section{Proofs}
\label{Sec: proofs}
\subsection{Proof of  Theorem \ref{THM: Strip}}
We split the arguments into four parts. We prove step by step the following propositions:
\par
1.
The convolution product (\ref{Eq: Convolution for strip}) is a continuous and bounded function in domain $\Pi^{\pi}$ for all bounded boundary functions $\varphi$ and $\varphi_1$ from the space $H^\lambda$.
\par
2.
The convolution product (\ref{Eq: Convolution for strip}) is a twice differentiable function in the interior of the domain $\Pi^{\pi}$ for all bounded boundary function $\varphi$ from $H^\lambda$.
\par
3.
The convolution product (\ref{Eq: Convolution for strip}) is the solution of the problem (\ref{Eq: Main Strip})
\par
4.
The solution of the problem (\ref{Eq: Main Strip}) is unique.

\subsubsection{Proposition 1}
We check the absolute convergence of integrals in the convolution (\ref{Eq: Convolution for strip}), implying  that the convolution is continuous function. Recall the formula from  (\ref{Eq: Convolution for strip})
\[
\varphi * G^{\pi} = \frac{1}{\pi} \int_{\mathbb{R}} \varphi(u)
\sum_{k=0}^{\infty} \frac{k \sin(ky) e^{-|x-u|\sqrt{k^2+r^2}}}{\sqrt{k^2+r^2}} du +
\Theta(-y) \int_{\mathbb{R}} \delta(x-u) \varphi(u) du.
\]
Representing the first summand as the sum of integrals taken over domains $(-\infty, x]$ and $[x, -\infty)$ and  using the change of  variables $x-u =\xi$ in the first integral and $u-x =\xi$ in the second, we get that
\[
\varphi * G = \frac{1}{\pi} \int_{0}^{\infty} \left( \varphi(x+\xi)- \varphi(x-\xi) \right)
\sum_{k=0}^{\infty} \frac{k \sin(ky) e^{-\xi \sqrt{k^2+r^2}}}{\sqrt{k^2+r^2}} d\xi +\Theta(-y)\varphi(x). \]
The following chain of estimates, where $\Gamma(\nu)$ is the Euler Gamma function (see \cite{Abramowitz}) and  $C$ denotes constants varying from step to step, completes the proof:
\[
\int_{0}^{\infty}\lbr| \lbr(\varphi(x+\xi) - \varphi(x-\xi) \rbr)
\sum_{k=0}^{\infty} \frac{k \sin(ky) e^{-\xi \sqrt{k^2+r^2}}}{\sqrt{k^2+r^2}} \rbr| d\xi  \leq
\left\{ {\begin{array}{l}
        |\sin(ky)| \leq 1; \\
        |\varphi(x+\xi) - \varphi(x-\xi)| \leq C\xi^\lambda.  \\
    \end{array} }
\rbr\} \leq
\]

\[
\leq \int_{0}^{\infty}\lbr| C \xi^\lambda
\sum_{k=0}^{\infty} \frac{k e^{-\xi \sqrt{k^2+r^2}}}{\sqrt{k^2+r^2}} \rbr| d\xi \leq
\left\{ {\begin{array}{l}
       \int_0^{\infty} x^{\nu-1} e^{-\mu x} dx  = \frac{\Gamma(\nu)}{\mu^\nu} , \\
        \quad {\rm Re}\mu,\, {\rm Re}\nu >0.  \\
    \end{array}} \right\}
\leq
\]

\[
\leq C \Gamma(1+\lambda) \sum_{k=0}^{\infty} \frac{k}{(r^2+k^2)^{1+\lambda/2}}
\leq C \sum_{k=0}^{\infty} \frac{1}{k^{1+\lambda/2}} <\infty.
\]

\subsubsection{Proposition 2}
We consider the Dirichlet problem for the equation $LV =0$ in the disk $\Omega$ centered at the origin with radius $\tilde{\rho}$.
\begin{equation}
\label{Eq: Dirichlet in the round}
    \left\{ {\begin{array}{l}
        LV(x,y) = 0, \quad (x,y) \in \Omega, \\
        V|_{\partial\Omega} = \varphi(x,y),  \\
    \end{array}} \right.
\end{equation}
First we prove the following lemma:
\begin{lemma}
\label{Lemma 1}
The problem (\ref{Eq: Dirichlet in the round}) has a classical solution for any continuous boundary function $\varphi(x,y)$.
\end{lemma}

\begin{proof}
In the polar coordinates $x =\rho \cos\theta$, $y = \rho \sin\theta$, the problem (\ref{Eq: Dirichlet in the round}) for $V(\rho, \theta)$ has the form
\begin{equation}
\label{Eq: Dirichlet in the round_ POLAR}
    \left\{ {\begin{array}{l}
        \diffp[2]{V}{\rho} + \frac{1}{\rho} \diffp{V}{\rho} + \frac{1}{\rho^2} \diffp[2]{V}{\theta} = 0, 
        \\
        V(\tilde{\rho}, \theta) = \psi(\theta). \\
    \end{array}} \right.
\end{equation}
where $\psi(\theta) = \varphi(x,y)|_{(x,y) \in \partial \Omega}$. We separate the variables and solve the Sturm--Liouville problem
\[
V(\rho,\theta) = \sum_{n=0}^{\infty} P_n(\rho) Q_n(\theta), \quad\quad \frac{P''_n +P'_n/\rho -r^2 P_n}{P_n /\rho^2} =\frac{Q''_n}{Q_n} = n^2, \quad n \geq 0.
\]
It is well known that the solution can be represented as follows
\begin{eqnarray*}
    V(\rho,\theta) &=& \frac{2}{\pi} \frac{I_0(r \rho)}{I_0(r\tilde{\rho})} \int_{-\pi}^{\pi} \psi(\xi) d\xi \\
    &&+\frac{1}{\pi} \sum_{n=0}^{\infty} \frac{I_n(r \rho)}{I_n(r\tilde{\rho})} \lbr( \sin{n \theta} \int_{-\pi}^{\pi} \sin(n\xi)\psi(\xi) d\xi +
    \cos{n \theta} \int_{-\pi}^{\pi} \cos(n\xi)\psi(\xi) d\xi\rbr)
\end{eqnarray*}
where $I_n(z)$ is the modified Bessel function (see \cite{Abramowitz}).

It is easy to show that $I_n(z)$ has the following properties:
\beq
\frac{I_n(z)}{I_n(\tilde{z})} \leq \left( \frac{z}{\tilde{z}} \right)^n,
\qquad \frac{I_{n+1}(z)}{I_n(\tilde{z})} \leq \left( \frac{z}{\tilde{z}} \right)^n \frac{z}{2},
\qquad \frac{I_{n-1}(z)}{I_n(\tilde{z})} \leq \left( \frac{z}{\tilde{z}} \right)^n \frac{2}{z}.
\enq
Indeed, using the  definition of the modified Bessel function $I_n (z)$, see \cite{Abramowitz}, we have:
\[
\frac{I_{n}(z)}{I_{n}(\tilde{z})} = \frac{\sum_{k=0}^{\infty} \frac{1}{k!(n+k)!}\lbr(\frac{z}{2}\rbr)^{n+2k}}
{\sum_{k=0}^{\infty} \frac{1}{k!(n+k)!}\lbr(\frac{\tilde{z}}{2}\rbr)^{n+2k}}  \leq
\lbr(\frac{z}{\tilde{z}} \rbr) ^n \frac{\sum_{k=0}^{\infty} \frac{1}{k!(n+k)!}\lbr(\frac{z}{2}\rbr)^{2k}}
{\sum_{k=0}^{\infty} \frac{1}{k!(n+k)!}\lbr(\frac{z}{2}\rbr)^{2k}} \leq
\lbr(\frac{z}{z} \rbr) ^n,
\]
\[
\frac{I_{n+1}(z)}{I_{n}(\tilde{z})} = \frac{\sum_{k=0}^{\infty} \frac{1}{k!(n+k+1)!}\lbr(\frac{z}{2}\rbr)^{n+1+2k}}
{\sum_{k=0}^{\infty} \frac{1}{k!(n+k)!}\lbr(\frac{\tilde{z}}{2}\rbr)^{n+2k}}  \leq
\lbr(\frac{z}{\tilde{z}} \rbr) ^n \frac{z}{2} \frac{\sum_{k=0}^{\infty} \frac{1}{k!(n+k)!}\lbr(\frac{z}{2}\rbr)^{2k}}
{\sum_{k=0}^{\infty} \frac{1}{k!(n+k)!}\lbr(\frac{z}{2}\rbr)^{2k}} \leq
\lbr(\frac{z}{z} \rbr) ^n \frac{z}{2},
\]
\[
\frac{I_{n-1}(z)}{I_{n}(\tilde{z})} = \frac{\sum_{k=0}^{\infty} \frac{1}{k!(n+k-1)!}\lbr(\frac{z}{2}\rbr)^{n-1+2k}}
{\sum_{k=0}^{\infty} \frac{1}{k!(n+k)!}\lbr(\frac{\tilde{z}}{2}\rbr)^{n+2k}}  \leq
\lbr(\frac{z}{\tilde{z}} \rbr) ^n \frac{2}{z} \frac{\sum_{k=0}^{\infty} \frac{1}{k!(n+k-1)!}\lbr(\frac{z}{2}\rbr)^{2k}}
{\sum_{k=0}^{\infty} \frac{1}{k!(n+k-1)!}\lbr(\frac{z}{2}\rbr)^{2k}} \leq
\lbr(\frac{z}{z} \rbr) ^n \frac{2}{z}.
\]

From here we immediately obtain the bounds for  $V(\rho,\theta)$ and $\diffp[2]{V(\rho,\theta)}{\theta}$:
\[
|V(\rho,\theta)| \leq C\sum_{n=0}^{\infty} \lbr( \frac{\rho}{\tilde{\rho}}\rbr)^n \leq {\rm Const}, \quad
\lbr|\diffp[2]{V(\rho,\theta)}{\theta} \rbr| \leq C\sum_{n=0}^{\infty} n^2 \lbr( \frac{\rho}{\tilde{\rho}}\rbr)^n \leq {\rm Const}.
\]

For the derivative $\partial V / \partial \rho$ we use the  properties of the modified Bessel functions, see \cite{Abramowitz}:
\[
I_{n-1} (r \rho) + I_{n+1} (r \rho) = 2\frac{d}{r d\rho} \lbr( I_n(r \rho) \rbr), \quad I_{n-1} (r \rho) - I_{n+1} (r \rho) = 2\frac{n}{r \rho} I_n(r \rho).
\]

Thus, the derivative $\partial V / \partial \rho$ is bounded because
\[
\lbr| \diffp{V}{\rho} (\rho, \theta) \rbr| \leq C \left( \sum_{n=0}^{\infty} \frac{I_{n-1}(r \rho)}{I_n(r \tilde{\rho})} +
\sum_{n=0}^{\infty} \frac{I_{n+1}(r \rho)}{I_n(r \tilde{\rho})} \right) \leq {\rm Const}.
\]
These relations mean that $V$, $\partial V  / \partial \theta $, and $\partial V /  \partial \rho $ are continuous functions. The arguments for the second derivative $ \partial^2 V / \partial \rho ^2 $ are similar.
\end{proof}
\par
Now we can complete the proof of Proposition 2. We take an arbitrary point $(x,y)$ from ${\rm int \,}\Pi^{\pi}$ and consider a disk centered in $(x,y)$ and contained in ${\rm int \,}\Pi^{\pi}$. Due to  Proposition 1 the convolution (\ref{Eq: Convolution for strip}) is continuous on the boundary of this disk. By Lemma \ref{Lemma 1} (\ref{Eq: Convolution for strip}) it is twice differentiable  at any point of its interior. Hence, the convolution (\ref{Eq: Convolution for strip}) is a twice differentiable function in the interior of $\Pi^{\pi}$.

\subsubsection{Proposition 3}
First, we check the boundary conditions
\[
V(x,y)|_{y=0} =\int_{\mathbb{R}} \varphi(u) G^{\pi}(x-u,0)du= \varphi(x), \quad V(x,y)|_{y=\pi} =\int_{\mathbb{R}} \varphi(u) G^{\pi}(x-u,\pi)du=0.
\]
It is easy to show that the Green function is a weak solution, i.e. we understand the equality $LG^{\pi} = 0$ in the following sense:
\[
LG^{\pi}= 0\ \Leftrightarrow\ \left \{\lbr(G^{\pi}, Lz \rbr)_{\Omega} = \iint_{\Omega} G^{\pi} Lz dxdy =0,
\quad \forall z \in \dot{C}^2, \quad \forall \Omega \subsetneq \Pi^\pi \right\},
\]
where $\dot{C}^2$ is the class of twice continuously differentiable finite functions (the class of test functions) and $\Omega$ is a compact sub-domain. Hence, for the convolution $\varphi * G^{\pi}$ we have
\[
\lbr(\varphi * G^{\pi}, Lz \rbr)_{\Omega} = \iint_{\Omega} \lbr(\varphi * G^{\pi}\rbr) Lz dxdy  = \varphi  * \iint_{\Omega} G^{\pi} Lz dxdy =0.
\]
It is well known that the elliptic operator is self-adjoint. Therefore

\[
\lbr((L(\varphi * G^{\pi}), \varphi \rbr)_{\Omega}  = \lbr(\varphi * G^{\pi}, L^*z \rbr)_{\Omega}
=\lbr(\varphi * G^{\pi}, L^*z \rbr)_{\Omega} = \lbr(\varphi * G^{\pi}, Lz \rbr)_{\Omega} =0.
\]
Due to Propositions 1 and 2 function the $L(\varphi * G^{\pi})$ is continuous and, therefore, bounded on the compact  $\Omega$. Hence, $ L(\varphi * G^{\pi}) =0$ because $\Omega$ is arbitrary.

\subsubsection{Proposition 4}
Suppose that we have two different bounded functions $V$ and $\tilde{V}$, both solving the problem (\ref{Eq: Main Strip}). Their difference $V-\tilde{V}$ is a bounded function solving the homogeneous problem
\begin{equation}
\label{Eq: Homogeneous Strip}
    \left\{ {\begin{array}{l}
        LV_0(x,y) = 0, \quad (x,y) \in {\rm int \,}\Pi^\pi, \\
        V_0(x,0) = 0, \quad x \in \mathbb{R},  \\
        V_0(x,\pi) = 0, \quad x \in \mathbb{R}.  \\
    \end{array}} \right.
\end{equation}

\begin{lem}
The solution of the problem (\ref{Eq: Homogeneous Strip}) is the infinite sum
\beq
\label{eq: solution of homogeneous strip}
V_0(x,y) =\sum_{k=0}^{\infty} \lbr(A_k e^{\sqrt{r^2+k^2}x} \sin ky + B_k e^{ - \sqrt{r^2+k^2}x} \sin ky \rbr).
\enq
\end{lem}
It is easy to show that all summands in (\ref{eq: solution of homogeneous strip}) are unbounded functions except zero (the eigenfunction corresponding to $k=0$). Hence, $V = \tilde{V}$.
\begin{proof}
We separate the variables $V_0 (x,y)=X(x)Y(y)$ and solve the Sturm--Liouville problem:
\[
\label{Eq: Strurm - Liouville Uniqueness}
YX''+Y''X-r^2XY = 0, \qquad \frac{X''-r^2X}{X}=\frac{-Y''}{Y}=\lambda^2.
\]
Spectrum  is $\lambda = \lambda_k = k$, $k \in \mathbb{N} $ and the eigenfunctions are $X_k(x)Y_k(y)=e^{\pm \sqrt{r^2+k^2}x} \sin ky$.
\end{proof}

\subsection{Proof of Theorem \ref{THM: Strip 2}}
We put $\xi= x-u$ and consider the part of the Green function $G^\pi$ given by the series:
\beq
\label{Eq: R series}
R(\xi,y,r) =\sum_{k=0}^{\infty} \frac{k \sin ky}{\sqrt{k^2+r^2} }e^{-\xi\sqrt{k^2+r^2}}.
\enq
In the  case $r=0$, see \cite{Gradshteyn} 1.445.1, we have:
\beq
\label{Eq: R series r =0}
R(\xi,y,0) = \frac{\sin y}{2\lbr(\cosh\xi-\cos y \rbr)}, \quad \diffp{R(\xi,y,0)}{\xi} = T(\xi,y,0) =  -\frac{\sin y \sinh \xi }{2 \lbr(\cosh\xi -\cos y  \rbr)^2}.
\enq
Using the formula
\[
\int_{1}^{\infty} e^{-k \xi t } J_0 (\xi r \sqrt{t^2-1}) dt = \frac{e^{- \xi \sqrt{k^2 + r^2}}}{\xi \sqrt{k^2+r^2}},
\]
 see \cite{Gradshteyn}, 6.646(1)), we have
\[
    R(\xi,y,r)= \xi \sum_{k=0}^{\infty} k \sin ky \int_{1}^{\infty} e^{-k \xi t} J_0 \lbr( \xi r \sqrt{t^2-1} \rbr) dt.
\]
The change of the  integration and the summation yields
\[
   R(\xi,y,r) = -\xi \int_{1}^{\infty} J_0 \lbr( \xi r \sqrt{t^2-1} \rbr) T(\xi t, y,0) dt =
   -\int_{1}^{\infty} J_0 \lbr( \xi r \sqrt{t^2-1} \rbr) d\left(R(\xi t, y,0) \right).
\]
Using the second formula in (\ref{Eq: R series r =0}), we finish the proof.

\subsection{Proof of Theorem \ref{THM: Plane}}
The proof consists of four parts:

\par
1.    The convolution product (\ref{Eq: Convolution for HalfPlane}) is a continuous and bounded function in the domain $\Pi^{\infty}$ for all bounded boundary functions $\varphi$ from  $H^\lambda$.
\par
2.    The convolution product (\ref{Eq: Convolution for HalfPlane}) has the continuous second-order partial derivatives in $\Pi^{\infty}/\partial{\Pi^{\infty}} $ for all bounded boundary functions $\varphi$ from $H^\lambda$.
\par
3.    The convolution product (\ref{Eq: Convolution for HalfPlane}) is the solution of the problem (\ref{Eq: Main HalfPlane})
\par
4.    The  solution of the problem (\ref{Eq: Main HalfPlane}) is unique.

We omit the proofs of Propositions 2 and 3 as the arguments are similar to those in the previous case.

\subsubsection{Proposition 1}
We make changes in the variables similar to those used above and get that

\[
G^{\infty}*\varphi =
\frac{1}{\pi} \iint_{\mathbb{R}^2_{+}}\lbr(\varphi(x+\eta) - \varphi(x-\eta) \rbr) \frac{\xi \sin(\xi y) e^{-\eta \sqrt{r^2+\xi^2}}}{\sqrt{r^2+\xi^2}} d\eta d\xi+ \Theta(-y)\varphi(x).
\]
Using the bounds relations we infer the absolute convergence of the integral:
\[
\iint_{\mathbb{R}^2_{+}} \lbr|\varphi(x+\eta) - \varphi(x-\eta) \rbr| \frac{\xi |\sin(\xi y)| e^{-\eta \sqrt{r^2+\xi^2}}}{\sqrt{r^2+\xi^2}} d\eta d\xi \leq
\]
\[
\leq
\left\{ {\begin{array}{l}
        |\sin(ky)| \leq 1; \\
        |\varphi(x+\xi) - \varphi(x-\xi)| \leq C\xi^\lambda  \\
    \end{array} }
\rbr\} \leq  C\iint_{\mathbb{R}^2_{+}}   \eta^{\lambda} \frac{\xi e^{\eta \sqrt{r^2+\xi^2}}}{\sqrt{r^2+\xi^2}} d\eta d\xi \leq
\]
\[
\leq
\left\{ {\begin{array}{l}
       \int_0^{\infty} x^{\nu-1} e^{-\mu x} dx  = \frac{\Gamma(\nu)}{\mu^\nu} , \\
        \quad { \rm Re \,} \mu, { \rm Re \,} \nu >0  \\
    \end{array}} \right\}
\leq C \Gamma(1+\lambda)\int_{0}^{\infty} \frac{\xi d\xi}{\lbr( r^2 + \xi^2 \rbr)^{1+\lambda/2}} \leq C \int_{0}^{\infty} \frac{d\xi}{\xi^{1+\lambda/2}} \leq C.
\]
\subsubsection{Proposition 4}
As in  the previous case we suppose that there are two different bounded solutions $V$ and $\tilde{V}$ of the problem (\ref{Eq: Main HalfPlane}). Then we have the homogeneous boundary problem for the difference $V_0 = V-\tilde{V}$.

\begin{equation}
\label{Eq: Homogeneous Plane}
    \left\{ {\begin{array}{l}
        \Delta V_0 - r^2 V_0 = 0, \quad (x,y) \in {\rm int \,} \Pi^{\infty}, \\
        V_0(x,0) = 0, \quad x \in \mathbb{R},  \\
        V_0(x,\infty) = 0, \quad x \in \mathbb{R}.  \\
    \end{array}} \right.
\end{equation}
We show that the solution of  (\ref{Eq: Homogeneous Plane}) can be represented in terms of the modified Bessel functions. After that we use the asymptotic of modified Bessel functions and show that only the zero function solves the problem (\ref{Eq: Homogeneous Plane}) in the class of bounded functions. The proof of the lemma below completes the proof of  Proposition 4.
\begin{lem}
The solution of the problem (\ref{Eq: Homogeneous Plane}) is the infinite sum
\beq
\label{Eq: solution of homogeneous plane}
V_0(x,y) = \sum_{n=0}^{\infty} \left( A_n I_n (r\nu) \sin ny   + B_n K_n (r\nu) \sin ny  \right).
\enq
This sum does not have bounded summands for $n\ge 1$.
\end{lem}
\begin{proof}
In the polar coordinates $x=\nu \cos\varphi $ and $y = \nu \sin \varphi$ we have the following Sturm--Liouville problem in separated variables $V=N(\nu)\Phi(\varphi)$:
\[
\frac{N''+N'/\nu -r^2 N}{N/\nu^2} = -\frac{\Phi''}{\Phi} = \lambda^2.
\]
For the phase component $\Phi$ we have the spectral problem
\begin{equation}
\label{Eq: Plane spectral angle}
    \left\{ {\begin{array}{l}
        \Phi'' +\lambda^2 \Phi= 0, \\
        \Phi(0) = 0,  \\
        \Phi(\pi)=0.
    \end{array}} \right.
\end{equation}
The eigenfunctions are $\Phi_k(\varphi) =\sin k\varphi $ and $\lambda= \lambda_k=k$ with $k \in \mathbb{N}$. For the radial component we have the modified Bessel equation
\[
N''_k+\frac{N'_k}{\nu}- r^2N_k -\frac{k^2}{\nu^2}N_k=0.\]
Therefore, the solution of the problem (\ref{Eq: Homogeneous Plane}) can be represented in the form from (\ref{Eq: solution of homogeneous plane}). The modified Bessel functions have
the well-known asymptotic, see \cite{Abramowitz},
\[
K_n(z) \thicksim \infty, \quad I_n(z) \thicksim 0, \quad (z \rightarrow 0),
\qquad
K_n(z) \thicksim 0, \quad I_n(z) \thicksim \infty, \quad (z \rightarrow \infty).
\]
Hence, in the sum we do not have bounded summands for $n\ge 1$.
\end{proof}

\

\smallskip
\noindent
{\bf Acknowledgement.} The research is funded by the grant of the Government of Russian Federation  $n^\circ$14.А12.31.0007.

\end{document}